\documentclass[12pt]{amsart}
\usepackage{latexsym}
\usepackage{graphicx}
\usepackage{amsfonts,amsmath,amsthm,amssymb,enumerate}
\usepackage{pb-diagram,pb-xy}
\usepackage{subfigure}
\usepackage[all]{xy}
\usepackage{mathrsfs}
\usepackage{graphics}
\usepackage{epsfig}
\usepackage{psfrag}
\usepackage{amsmath}
\usepackage{amssymb}
\usepackage{mathrsfs}
\usepackage{epstopdf}

\usepackage[usenames,dvipsnames]{color}
\newtheorem{theorem}{Theorem}[section]
\newtheorem{lemma}[theorem]{Lemma}
\newtheorem{prop}[theorem]{Proposition}

\newtheorem{definition}[theorem]{Definition}

\newenvironment{remark}{\noindent \textbf{Remark}.}{\hfill $\square$}

\newcommand{\bb}{\mathbb}
\newcommand{\cal}{\mathcal}

\numberwithin{equation}{section}

\newcommand{\Hom}{\rm Hom}

\def\Ext(#1,#2,#3){\mbox{Ext}^{#1}(#2,#3)}
\def\Hom(#1,#2){\mbox{Hom}(#1,#2)}
\def\hom(#1,#2){\mbox{hom}(#1,#2)}
\def\ext(#1,#2,#3){\mbox{ext}^{#1}(#2,#3)}

\def\C{\mathbb{C}}
\def\P{\mathbb{P}}

\def\O{\mathcal{O}}

\def\a'{\alpha}
\def\t'{\beta}
\def\g'{\gamma}
\def\l'{\lambda}

\def\pin(#1,#2){\left\langle #1,#2 \right\rangle}
\def\pti(#1,#2,#3){\left\langle #1,#2,#3 \right\rangle}
\def\suma(#1,#2,#3){\sum\limits_{#1=#2}^#3}

\title{The birational geometry of moduli space of sheaves on the projective plane }
\author{Aaron Bertram}
\address{University of Utah, Department of Mathematics, Salt Lake City, UT 84112}
\email{bertram@math.utah.edu}

\author{Cristian Martinez}
\address{University of Utah, Department of Mathematics, Salt Lake City, UT 84112}
\email{martinez@math.utah.edu}

\author{Jie Wang}
\address{University of Georgia, Department of Mathematics, Athens, GA 30602}
\email{jiewang@math.uga.edu}

\subjclass[2000]{}
\begin{document}
\maketitle

\begin{abstract} 
 We describe a close relation between wall crossings in the birational geometry of moduli space of Gieseker stable sheaves $M_H(v)$ on $\bb{P}^2$  and mini-wall crossings in the stability manifold $Stab(D^b(\bb{P}^2))$.
\end{abstract}

\bigskip

\section{Introduction.}
Let $(X,H)$ be an anti-canonically polarized smooth Del Pezzo surface and $M_H(v)$ be the coarse moduli space parametrizing $S$-equivalent classes of Gieseker semi-stable sheaves on $X$ of topological type $v$. In this paper, we study the birational geometry of $M_H(v)$ with emphasis on the case $X=\bb{P}^2$. 

The general philosophy is that the birational geometry of $M_H(v)$ is closely related to the birational geometry of $X$. So if $-K_X$ is ample, we hope $-K_{M_H(v)}$ is ample as well. This turns out to be not quite right but is close. We first show that if $v$ is a primitive topological type such that $M_H(v)$ is non-empty and irreducible, then $M_H(v)$ is smooth and $-K_{M_H(v)}$ is big and nef. In particular, $M_H(v)$ is a Mori dream space. Therefore $M_H(v)$ provide a rich class of examples of the so called weak Fano varieties, whose classification theory is highly interesting in its own right.  We achieve the statement by showing that the first Chern class of $-K_{M_H(v)}$ is equal to the first Chern class of the determinant line bundle constructed by J. Li in \cite{L1} which gives a contracting morphism from $M_H(v)$ to the Donaldson-Uhlenbeck compactification of the moduli space of $\mu$-stable vector bundles as defined in Gauge theory. The computation of the Chern classes of $-K_{M_H(v)}$ seems to be well known among the experts and could be found, for instance, in \cite{HL}.

 As the name ``Mori dream space'' suggests, the Mori theory of $M_H(v)$ behaves as nicely as it could be. There exists a a polyhedral chamber decomposition of its pseudo-effective cone $\overline{NE}^1(M_H(v))$. These
chambers are known as the Mori chambers. Specifically, If $L$ is a effective line bundle on a Mori dream space $M$, then its section ring 
$$R(M,L):=\oplus_{n\ge0}H^0(M,L^n)$$
is finitely generated. Thus the rational map defined by the linear series $|L^n|$ stabilizes to some rational map
$$\Phi_L:M\dashrightarrow Proj R(M,L)$$
for all large divisible $n$. Two line bundles $L_1$ and $L_2$ are said to be Mori equivalent if $\Phi_{L_1}=\Phi_{L_2}$. This equivalence relation naturally extends to $N^1(M)_{\bb{R}}$ and a Mori chamber is just the closure of an equivalence class in $N^1(M)_{\bb{R}}$ whose interior is open in $N^1(M)_{\bb{R}}$. These chambers are polyhedral and in one-to-one correspondence with birational contractions of $M$ having $\bb{Q}$-factorial image by associating each chamber $\cal{P}$ the birational contraction
$$\Phi_L:M\dashrightarrow Proj R(M,L)$$
for some $L$ in the interior of $\cal{P}$. 

On the other hand, let $\cal{D}=\cal{D}^b(coh(X))$ be the bounded derived category of coherent sheaves on $X$. Bridgeland showed that the space of stability conditions on $\cal{D}$ is a complex manifold. In this paper, we consider a slice of the stability manifold parametrized by the upper half plane $s+\sqrt{-1}t$, $t>0$. For each $(s,t)$, there is an abelian subcategory $\cal{A}_s$ of $\cal{D}$ 
that forms the heart of a t-structure on $\cal{D}$ and a central charge $Z_{s,t}$ such that the pair $(\cal{A}_s,Z_{s,t})$ is a Bridgeland stability condition. Abramovich and Polishchuk \cite{AP} have constructed moduli stacks $\cal{M}_{s,t}(v)$ parametrizing Bridgeland semi-stable objects with topological type $v$.

When $X=\bb{P}^2$, it is proved in \cite{ABCH} that the moduli spaces of Bridgeland stable objects can be interpreted as the moduli spaces of quiver representations with respect to some polarization and can be constructed by Geometric Invariant Theory. In particular, there exist projective coarse moduli spaces. It is also shown that when $s<0$ and $t>>0$, the coarse moduli space $M_{s,t}(v)$ is isomorphic to $M_H(v)$. As we decrease $t$, $M_{s,t}(v)$ changes. Thus we obtain a wall and chamber decomposition of the $(s,t)$-plane into chambers which the corresponding moduli space are isomorphic. Crossing a Bridgeland wall means $M_{s,t}(v)$ changes.

The birational geometry of Hilbert scheme of $n$ points on $\bb{P}^2$, i.e. $v=(1,0,-n)$, has been extensively studied in \cite{ABCH}.\footnote{The birational geometry of the Hilbert scheme of points on a Del Pezzo surface has also been studied in \cite{BC} recently.} There a one-to-one correspondence of the wall and chamber structures between $\overline{NE}^1(\bb{P}^{2[n]})$ and certain region in the Bridgeland (s,t)-plane has been described. In this paper, we try to give a systematic interpretation why such phenomenal should happen for general $v$ on $\bb{P}^2$.

For $\bb{P}^2$, the fundamental reason is that there is a canonical choice of a determinant line bundle $\lambda_{s,t}$ on $M_{s,t}(v)$ which coincide with the ample line bundle in the GIT construction of the moduli of quiver representations. We prove that this canonical choice of determinant line bundle naturally determines a polarization $\vec{a}$ such that a complex of topological type $v$ is $(s,t)$ (semi-)stable if and only if it is quiver (semi)-stable with respect to the polarization $\vec{a}$ (c.f. section 4). Thus $\lambda_{s,t}$ coincide with the ample linearization in the GIT construction and therefore ample. This canonical choice of ample line bundle also turns out to coincide with the natural determinant line bundle associated to a stability condition $\sigma$ constructed by Bayer and Macri in \cite{BM}. Bayer and Macri proved that their line bundle is ample when $\sigma$ is a generic stability condition on the derived category of a K3 surface. In our situation, the determinant line bundles $\lambda_{s,t}$ is automatically ample because of the GIT construction. 

Ideally, as we decrease $t$ (fixing $s$), the pull-back of $\lambda_{s,t}$ to $M_H(v)$ is a path in $NS^1(M_H(v))_{\bb{R}}$, so the various $M_{s,t}$ we obtain as we decrease $t$ should correspond to running a directed MMP along the path in $NS^1(M_H(v))_{\bb{R}}$. However, due to the complicatedness of the affine scheme of which we take quotient in the GIT construction, we do not even know in general $M_{s,t}(v)$ is irreducible. Therefore in our description, we have to pass to the main component $M^P_{s,t}(v)$ of $M_{s,t}(v)$ whose generic point is parametrizing a sheaf (instead of a complex of sheaves). 

We summarize our main results as

\begin{theorem} Let $X=\bb{P}^2$ polarized by the hyperplane class $H$ and $M_H(v)$ be the coarse moduli space parametrizing S-equivalence classes of Gieseker semi-stable torsion free sheaves of primitive topological type $v=(r,c_1,ch_2)$.
\begin{enumerate}\item If $M_H(v)$ is non-empty, then it is a smooth weak Fano variety ( i.e $-K_{M_H(v)}$ is big and nef) of Picard rank at most $2$. In particular, it is a Mori dream space. 
\item Starting from $t>>0$ and decreasing $t$ corresponds to running a directed MMP on $M_H(v)$. As long as the generic point of the exceptional loci of each contraction is a sheaf \footnote{It is not clear if there exists a Bridgeland wall on which no sheaves are destabilized.}, each birational model (both in the interior of the Mori chamber and on the wall) we get in the directed MMP is isomorphic to the normalization the of main component of the Bridgeland moduli in the corresponding Bridgeland chamber and wall. 
\end{enumerate}
\end{theorem}

The organization of this paper is as follows. In section 2, we recall basic facts about Bridgeland stability conditions and introduce a complex plane worth of stability conditions that arise in our study of the birational geometry of $M_H(v)$. In section 3, we review the determinant line bundle construction on $M_H(v)$ and present a computation of the first Chern class of $-K_{M_H(v)}$ as an application of the Grothendieck-Riemann-Roch theorem. In section 4, specialize to the $\bb{P}^2$ case. We first construct a natural determinant line bundle $\lambda_{s,t}$ on $M_{s,t}(v)$. For each $(s,t)$, we find a suitable polarization $\vec{a}=(a_0,a_1,a_2)$ depending canonically on $(s,t)$ and prove that a complex of topological type $v$ is $(s,t)$-stable if and only if it is quiver stable with respect to polarization $\vec{a}$. As $t$ changes, the polarization $\vec{a}$ also changes, so this problem can be also viewed as a variation of GIT problem.  In section 5, we run the directed MMP on $M_H(v)$ by decreasing $t$ and interpret the various birational models we get in the process as the main component of the Bridgeland moduli spaces. Finally, in section 6, we describe explicitly the flips for the topological types $(0,4,-4)$ and $(0,5,-\frac{15}{2})$ by using the stratification of the Gieseker moduli given by Drezet and Maican \cite{DM} and by Maican \cite{M5} and give similar stratifications for the Bridgeland moduli spaces.\

\section{Preliminaries}\label{sec:prelim}
In this section, we review some basic facts concerning Bridgeland stability conditions.

Let $D^b(X)$ denote the bounded derived category of coherent sheaves on a smooth projective surface $X$. 

\begin{definition}
A numerical pre-stability condition $\sigma$ on $D^b(X)$ consists of a pair $(Z,\cal{A})$, where 
$$Z=-d+\sqrt{-1}\ r:\ K(X)_{num}\rightarrow\bb{C}$$
is a group homomorphism called central charge and $\cal{A}\subset D^b(X)$ is the heart of a $t$-structure, satisfying the following properties:
\begin{enumerate}
\item $r(E)\ge0$ for all $E\in\cal{A}$;
\item if $r(E)=0$ and $E\in\cal{A}$ nonzero, then $d(E)>0$.
\end{enumerate} 
\end{definition}

We can use them to define a notion of slope-stability via the slope $\mu(E)=\frac{d(E)}{r(E)}$. $\mu$ takes value in $(-\infty,+\infty]$.

\begin{definition} An object $E\in\cal{A}$ is stable (resp. semistable) if for any subobject $F$ of $E$ in $\cal{A}$,
$$\mu(F)<\mu(E)\ (resp. \mu(F)\le\mu(E)).$$
A pre-stability condition $\sigma=(Z,\cal{A})$ is a stability condition if any nonzero object $E\in\cal{A}$ 
admits a finite filtration:
$$0\subsetneq E_0\subsetneq E_1\subsetneq...\subsetneq E_n=E$$
uniquely determined by the property that each $F_i:=\frac{E_i}{E_{i-1}}$ is semistable and 
$$\mu(F_0)>\mu(F_1)>...>\mu(F_n).$$
This property is called the Harder-Narasimhan property.

\end{definition}

We will also use the notation of slicing introduced by Bridgeland.
For a stability condition $\sigma$ and a real number $\phi\in(0,1]$, we define a full abelian subcategory $\cal{P}(\phi)$ of $D^b(X)$ consisting of semistable objects of phase $\phi$. The relation between phase and slope is given by
$$\mu=-\cot(\pi\phi).$$
We then inductively define the category $\cal{P}(\phi)$ of semistable objects of arbitrary phase $\phi$ by enforcing
$$\cal{P}(\phi+1)=\cal{P}(\phi)[1].$$
Finally denote $\cal{P}(a,b]$ be the full subcategory consisting of objects whose semistable factors have phase in $(a,b]$. With this notation, the heart $\cal{A}=\cal{P}(0,1]$.

It turns out that on a smooth projective surface, we can not take $coh(X)$ as the heart of any stability condition. The next simplest heart, which we will use extensively in this paper, comes from tilting a torsion pair.

Fix a polarization $H$ on $X$.  The Mumford slope (w.r.t $H$)
$$\mu^M(E)=\frac{c_1(E)\cdot H}{rk(E)}$$ does not give a stability condition on $coh(X)$
since
$$c_1(\bb{C}_p)=rk(\bb{C}_p)=0.$$

Nevertheless, it has a (weak) Harder-Narasimhan property:
Any coherent sheaf $E$ admits a filtration
$$E_0\subsetneq E_1\subsetneq...\subsetneq E_n=E$$
where $E_0$ is the torsion subsheaf of $E$ and for $i>0$, the subquotients $F_i:=\frac{E_i}{E_{i-1}}$ are Mumford semi-stable torsion free sheaves of strictly decreasing Mumford slopes.

For any real number $s$, consider the torsion pair $(\cal{Q}_s,\cal{F}_s)$:
\begin{enumerate}
\item $Q\in\cal{Q}_s$ if $Q$ is torsion or if each $\mu^M_i>s$ in the HN-filtration of $Q$.
\item $F\in\cal{F}_s$ if $F$ is torsion free and each $\mu^M_i\le s$ in the HN-filtration of $F$.
\end{enumerate}

Each pair ($\cal{Q}_s,\cal{F}_s$) are full subcategories of $D^b(X)$ satisfying
\begin{enumerate}
\item $Hom(Q,F)=0$ for any $Q\in\cal{Q}_s$, $F\in\cal{F}_s$;
\item Every coherent sheaf $E$ fits in a (unique up to isomorphism) short exact sequence
$$0\rightarrow Q\rightarrow E\rightarrow F\rightarrow0.$$
\end{enumerate}

The heart $\cal{A}_s$ of the t-structure obtained by tilting the torsion pair $(\cal{Q}_s,\cal{F}_s)$ is the extension closure of $\langle\cal{Q}_s,\cal{F}_s[1]\rangle$ consisting of
$$\cal{A}_s=\{E\in D^b(X)\ |\ H^{-1}(E)\in\cal{F}_s, H^0(E)\in\cal{Q}_s\ \text{and }H^i(E)=0\ \text{otherwise} \}$$ 

\begin{theorem} \label{existence}$($\cite{AB},\cite{B2}$)$For each $s\in\bb{R}$ and $t>0$, choose an $\bb{R}$-divisor $D$ such that $D\cdot H=s$, then the central charge
\begin{eqnarray}Z(E)&=&-\int_X e^{-(D+\sqrt{-1} tH)}ch(E)\nonumber\\
&=&-\int_X [ch_2(E)-D\cdot c_1(E)+\frac{1}{2}r(D^2-t^2H^2)+\sqrt{-1}t(-c_1(E)\cdot H+rsH^2)]\nonumber
\end{eqnarray}
on $\cal{A}_s$ defines a stability condition.
\end{theorem}

For each stability condition, we could consider the moduli stack of semistable objects with fixed topological invariants. In the case $X=\bb{P}^2$, the first author proved that these moduli spaces coincide with quiver moduli space in the sense of King \cite{K}. In particular, there exists a projective  coarse moduli space as GIT quotient of some highly reducible affine scheme. We still do not know if these coarse moduli spaces are irreducible or not.

So let $X=\bb{P}^2$ for the rest of this section and $H$ be the hyperplane class. Take $D=sH$ in theorem \ref{existence} and think of $ch(E)=(r,c_1,d)$ as numbers. The central charge becomes
$$Z_{s,t}(E)=-(d-sc_1+\frac{1}{2}r(s^2-t^2))+\sqrt{-1}t(c_1-rs)$$
and the slope function becomes
$$\mu_{s,t}(E)=\frac{d-sc_1+\frac{1}{2}r(s^2-t^2)}{t(c_1-rs)}$$

The $(s,t)$ upper half plane is a slice of the stability manifold.
 
 \subsection{Quiver moduli}
 
 Every $(s,t)$-moduli space is a quiver moduli with respect to some polarization (c.f. \cite{ABCH}). For each integer $k\in\bb{Z}$ the three objects:
 $$\cal{O}_{\bb{P}^2}(k-2)[2],\ \cal{O}_{\bb{P}^2}(k-1)[1],\ \cal{O}_{\bb{P}^2}(k)$$
 form an ``Ext-exceptional'' collection. The extension closure
 $$\cal{A}(k):=\langle\cal{O}_{\bb{P}^2}(k-2)[2],\ \cal{O}_{\bb{P}^2}(k-1)[1],\ \cal{O}_{\bb{P}^2}(k)\rangle$$
 is the heart of a t-structure.
 
 The objects of $\cal{A}(k)$ are complexes $E$:
 $$\bb{C}^{n_0}\otimes\cal{O}_{\bb{P}^2}(k-2)\rightarrow\bb{C}^{n_1}\otimes\cal{O}_{\bb{P}^2}(k-1)\rightarrow\bb{C}^{n_2}\otimes\cal{O}_{\bb{P}^2}(k)$$
 
 A subobject of $E$ in $\cal{A}(k)$ is a sub-complex of $E$ of the same form but with dimension vectors $(b_0,b_1,b_2)$ where $b_i\le n_i$ for all $i$. Thus, quite unlike the category of coherent sheaves, there are only finitely many possible invariants for subobjects of an object with given invariants.
 
 Fix a dimension vector $\vec{n}=(n_0,n_1,n_2)$ and a triple of integers $\vec{a}=(a_0,a_1,a_2)$ such that
 $$a_0n_0+a_1n_1+a_2n_2=0,$$
 
 We say an object $E\in\cal{A}(k)$ with dimension vector $\vec{n}$ is (semi)stable with respect to the polarization $\vec{a}$, if for any nontrivial subobject $F$ of $E$ with dimension vector $\vec{b}=(b_0,b_1,b_2)$, $b_i\le n_i$,
 $$a_0b_0+a_1b_1+a_2b_2>(\ge)0.$$
 
 King constructed the coarse moduli spaces of quiver representations as the GIT quotient of an affine variety.

 On the quiver moduli space, there is a natural ample line bundle defined as follows. A family of complexes on $\bb{P}^2$ parametrized by a scheme $S$ is a complex:
$$
U(k-2)\longrightarrow V(k-1)\longrightarrow W(k)
$$
 on $S\times\bb{P}^2$, where $U,V,W$ are vector bundles of ranks $n_0,n_1,n_2$ pulled back from $S$, twisted, respectively by the pulled backs of $\cal{O}_{\bb{P}^2}(k-2),\cal{O}_{\bb{P}^2}(k-1),\cal{O}_{\bb{P}^2}(k)$. 
 
 In this setting, the determinant line bundle on $S$ 
 $$(\wedge^{n_0}U)^{a_0}\otimes(\wedge^{n_1}U)^{a_1}\otimes(\wedge^{n_2}U)^{a_2}$$ is the pull back of the ample line bundle on the moduli stack of complexes that restricts to the ample line bundle on the moduli space of semi-stable complexes determined by Geometric Invariant Theory.

\section{The Gieseker moduli} \label{sec:G} 

Now let $X$ be a smooth Del Pezzo surface and use the anti-canonical bundle $H=-K_X$ as the polarization. For the rest of the paper, Gieseker or Mumford stability are all with respect to this polarization. Fix a primitive  topological type 
 $v=(r,c_1,ch_2)\in H^*(X,\bb{Q})_{alg}$. Consider the coarse moduli space $M(v)$ parametrizing $S$-equivalent classes of Gieseker semi-stable torsion free sheaves \footnote{Much of the theory in sections 3, 4, 5 will work for torsion sheaves of pure dimension $1$. See section 6 for some explicit examples in the torsion case.} on $X$ with topological type $v$. Suppose $M(v)$ is nonempty and irreducible (e.g. $v=(1,0,-n)$ is the case of Hilbert scheme of $n$-point on $X$, when $X=\bb{P}^2$, $M(v)$ is always irreducible, c.f. \cite{LP} chapter 17). We will show in this section that $M(v)$ is smooth and weak Fano (i.e. $-K_{M(v)}$ is big and nef). In particular, by \cite{BCHM}, $M(v)$ is a Mori dream space and  there is a finite rational polyhedra decomposition of the pseudo-effective cone $\overline{NE}^1(M(v))\subset N^1(M(v))_{\bb{R}}$  according to the stable base locus of the divisors. The results in this section are well known to the experts and are implicitly comtained in the standard references \cite{HL} and \cite{LP}.
 
 \begin{remark}
 The condition that $v$ being primitive means $g.c.d. (r, c_1\cdot H, \chi)=1$ where $\chi=r\chi(\cal{O}_X)-c_1\cdot K_X/2+ch_2$ is the Euler characteristic of the sheaf. This condition implies that there are no strictly semistable sheaves, in other words $M(v)^s=M(v)$. $v$ being primitive also implies that there exists a universal sheaf $\cal{E}$ on $M(v)$ (c.f. \cite{HL} Cor 4.6.7).
 \end{remark}

\subsection{Smoothness and dimension}
Let $F$ be a stable sheaf on $X$ with topological type $v$. From deformation theory we know that
$\rm{Ext}^1_{\cal{O}_X}(F,F)$ is the tangent space for the deformation functor of $F$ and $\rm{Ext}^2_{\cal{O}_X}(F,F)$ is an obstruction space. Since $F$ is stable and $X$ is Del Pezzo, 
\begin{eqnarray}\label{obstruction}\rm{Ext}^2_{\cal{O}_X}(F,F)^{\lor}\cong\rm{Hom}_{\cal{O}_X}(F,F\otimes_{\cal{O}_X}K_X)=0.
\end{eqnarray}
This means that $M^s(v)=M(v)$ is smooth of dimension 
$$\rm{ext}^1_{\cal{O}_X}(F,F)=1-r^2-2rch_2+c_1^2$$
 by Riemann-Roch.

\begin{remark} If $F$ is strictly semistable, we can not conclude $M(v)$ is smooth at $[F]$ even if $\rm{Ext}^2_{\cal{O}_X}(F,F)=0$. The issue here is that $M(v)$ is just a coarse moduli space.
\end{remark}

\subsection{Determinant line bundles on $M(v)$.}
In this subsection, we briefly review the determinant line bundle construction on $M(v)$. We will describe a general method for associating to a flat family of coherent sheaves a determinant line bundle on the base of the family. We also describe a particular determinant line bundle $\cal{L}_1$ on $M(v)$ 
such that the linear series $|\cal{L}_1^m|$ for large $m$ contracts certain parts of the moduli space and defines a morphism from $M(v)$ to the Donaldson-Uhlenbeck compactification $M^{\mu ss}(v)$ of the moduli space of $\mu$-stable vector bundles. This subsection is based on the work of J. Le Potier and J. Li. We refer to \cite{HL} chapter 8 for an excellent exposition on this subject.

The Grothendieck group $K(X)$ of coherent sheaves on $X$ becomes a ring with $1=[\cal{O}_X]$ and multiplication $[F_1]\cdot[F_2]=[F_1\otimes^L F_2]$. There is also a natural pairing $\chi$ on $K(X)$ defined to be $\chi([F_1],[F_2])=\chi([F_1]\cdot[F_2])=\int_{X}ch(F_1) ch(F_2) td(X)$.  Since $X$ is a Del Pezzo surface, the Chern character map $ch: K(X)_{\bb{Q}}\rightarrow H^*(X,\bb{Q})_{alg}$ is an isomorphism and $\chi$ is a nondegenerate pairing on $K(X)_{\bb{Q}}$. Due to this isomorphism, we will occasionally abuse the notation by thinking of $\chi$ as a nondegenerate pairing on $H^*(X)_{alg}$ or sometimes even writing $\chi(u, v)$ for $u\in K(X)$, $v\in H^*(X)_{alg}$.

Let $\cal{E}$ be a flat family of sheaves of topological type $v$ on $X$ parametrized by $S$. Denote $[\cal{E}]$ its class in $K^0(X\times S)$. Denote the projection from $X\times S$ to $X$ and $S$ as in the diagram below.
$$
\begin{diagram}
\node{X\times S}\arrow{s,r}{p}\arrow{e,t}{q}\node{X}\\
\node{S}\node{}
\end{diagram}
$$                                                                                                                                                                                                  

Notice that $p$ is a smooth morphism so $p_!: K^0(X\times S)\rightarrow K^0(S)$ is well defined. 
\begin{definition}\label{def3.1}
Define $\lambda_{\cal{E}}: K(X)\longrightarrow Pic(S)$ be the composition of the homomorphisms:
$$
\begin{diagram}\dgARROWLENGTH=1.5em
\node{K(X)}\arrow{e,t}{q^*}\node{K^0(X\times S)}\arrow{e,t}{\cdot[\cal{E}]}\node{K^0(X\times S)}\arrow{e,t}{p_!}\node{K^0(S)}\arrow{e,t}{det}\node{Pic(S).}
\end{diagram}
$$
\end{definition}
Notice that $\lambda_{\cal{E}}$ is just the 'Fourier-Mukai transform' with kernel $\cal{E}$ on the $K$-group level, composed with the determinant homomorphism which associates to a finite complex of locally free sheaves $F^{\bullet}$ on $S$  its determinant line bundle  $\otimes_i\ (detF^i)^{(-1)^i}$.

If $L$ is a line bundle on $S$, it is easy to check that 
\begin{eqnarray}\label{universal}\lambda_{\cal{E}\otimes p^*L}(u)\cong\lambda_{\cal{E}}(u)\otimes L^{\chi(u, v)}.
\end{eqnarray} 

We now apply this construction to the universal sheaf $\cal{E}$ on $X\times M(v)$. The universal sheaf is only well defined up to tensoring with the pull back of a line bundle from the base. If we choose $u\in v^{\perp}\subset K(X)$ with respect to $\chi$, by (\ref{universal}), $\lambda_{\cal{E}}(u)$ will not depend on the ambiguity of the choice of the universal sheaf and therefore yields a line bundle on $M(v)$. We will simply write $\lambda(u)$ for this determinant line bundle on $M(v)$. 

\begin{remark}  There is no universal sheaf on the coarse moduli space $M(v)$ in general, the determinant line bundle is just a line bundle on the moduli stack. This line bundle, however, always descends to $M^s(v)$ for $u\in v^{\perp}$. If $M^s(v)\subsetneq M(v)$, one needs to put extra conditions on $u$ to guarantee that this line bundle descends to the whole coarse moduli space $M(v)$. (c.f \cite{HL} Theorem 8.1.5).
\end{remark}

There are two distinguished determinant line bundles on $M(v)$ given by taking
\begin{eqnarray}\label{generator}u_i=-r\cdot h^i+\chi(v,h^i)[\bb{C}_x]\in v^{\perp},\  i=0,1
\end{eqnarray}
where $h=[\cal{O}_H]\in K(X)$, $x\in X$ and $\cal{L}_i:=\lambda(u_i)$ does not depend on the choice of $x$.

It is proved in \cite{L1} that for large $m$, the linear systerm $|\cal{L}_1^m|$ is base point free and gives a morphism from $M(v)$ to the Donaldson-Uhlenbeck compactification $M^{\mu ss}(v)$ of the moduli space of $\mu$- semistable sheaves. Two stable torsion free sheaves $F_1$, $F_2$ defines the same point in $M^{\mu ss}(v)$ if and only if $F_1^{**}\cong F_2^{**}$ and $F_1^{**}/F_1$ has the same length as $F_2^{**}/F_2$ and gives the same point in the suitable symmetric product of $X$. As a consequence, $\cal{L}_1$ is big and nef (but not ample). 

\subsection{The canonical class of $M(v)$}

In this subsection, we will prove that the anti-canonical bundle $-K_{M(v)}$ is numerically equivalent to $\cal{L}_1$ and therefore is big and nef. Let $\cal{E}$ be a universal sheaf on $M(v)$ and $p$ be the projection from $X\times M(v)$ to $M(v)$. Denote 
$$\cal{H}om_p(\cal{E},-)=p_*\circ\cal{H}om(\cal{E},-):Coh(X\times M(v))\longrightarrow Coh(M(v))$$ the relative Hom functor and 
$$\cal{E}xt_p^{\bullet}(\cal{E},-)=Rp_*\circ R\cal{H}om(\cal{E},-)$$ its derived functor. The Kodaira-Spencer map naturally indentifies tangent bundle $T_{M(v)}$ with $\cal{E}xt^1_p(\cal{E},\cal{E})$, which can also be described as the sheaf associated to the presheaf
$$U\longrightarrow \rm{Ext}^1(\cal{E}|_{p^{-1}(U)},\cal{E}|_{p^{-1}(U)}).$$

It suffices to prove that $c_1(\cal{E}xt^1_p(\cal{E},\cal{E}))=c_1(\cal{L}_1)$. This is an application of the Grothendieck-Riemann-Roch theorem.

\begin{lemma} \label{lemma3.2}$\cal{H}om_p(\cal{E},\cal{E})\cong\cal{O}_{M(v)}$, $\cal{E}xt^2_p(\cal{E},\cal{E})=0$.

\end{lemma}
\begin{proof}Since $\cal{E}$ restricts to the fiber of $p$ is stable and there is a nowhere vanishing section of $\cal{H}om_p(\cal{E},\cal{E})$, namely the identity map on the fibers of $p$, the first statement follows. The second statement follows from (\ref{obstruction}).
\end{proof}

 In the Grothendieck group $K(M(v))$,
$$[\cal{E}xt^{\bullet}_p(\cal{E},\cal{E})]=[\cal{H}om_p(\cal{E},\cal{E})]-[\cal{E}xt^1_p(\cal{E},\cal{E})]+[\cal{E}xt^2_p(\cal{E},\cal{E})].$$
 By lemma \ref{lemma3.2},
 $$c_1(\cal{E}xt^1_p(\cal{E},\cal{E}))=-c_1(\cal{E}xt^{\bullet}_p(\cal{E},\cal{E}))=-c_1(Rp_*(\cal{E}^{\lor}\otimes^L\cal{E}))$$
 where $\cal{E}^{\lor}$ stands for the derived dual $R\cal{H}om^{\bullet}(\cal{E},\cal{O}_{X\times M(v)})$.

Apply Grothendieck-Riemann-Roch to the product family $X\times M(v)$, we have equalities in $H^*(X,\bb{Q})_{alg}$,

\begin{eqnarray}c_1(Rp_{*}(\cal{E}^{\lor}\otimes^L\cal{E}))&=&p_*\{(ch(\cal{E}^{\lor}\otimes^L\cal{E})\cdot q^*td(X))\}_3\nonumber\\
&=&p_*\{(ch(\cal{E})^{\lor}\cdot ch(\cal{E})\cdot q^*td(X))\}_3\nonumber\\
&=&p_*\{(r,-c_1(\cal{E}),ch_2(\cal{E}),-ch_3(\cal{E}))\cdot(r,c_1(\cal{E}),ch_2(\cal{E}),ch_3(\cal{E}))\cdot q^*(1,\frac{-K_X}{2},\chi(\cal{O}_X))\}_3\nonumber\\
&=&p_*\{(r^2,0,2rch_2(\cal{E})-c_1^2(\cal{E}),0)\cdot q^*(1,\frac{-K_X}{2},\chi(\cal{O}_X))\}_3\nonumber\\
&=&p_*\{(2rch_2(\cal{E})-c_1^2(\cal{E}))\cdot \frac{-q^*K_X}{2}\}\nonumber
\end{eqnarray}

Here $\{\}_3$ means taking complex degree 3 part of a cohomology class and $p_*$ is the Gysin map. 

On the other hand,
\begin{eqnarray}ch(u_1)&=&-r(0,H,-\frac{{H}^2}{2})+(c_1\cdot H)(0,0,1)\nonumber\\
&=&(0,-rH,\frac{rH^2}{2}+c_1\cdot H)\nonumber
\end{eqnarray}
and
\begin{eqnarray}c_1(\lambda(u_1))&=&c_1(p_!(q^*(u_1)\cdot[\cal{E}]))\nonumber\\
&=&p_*\{q^*ch(u_1)\cdot ch(\cal{E})\cdot q^*td(X)\}_3\nonumber\\
&=&p_*\{q^*(0,-rH,\frac{rH^2}{2}+c_1\cdot H)\cdot q^*(1,\frac{H}{2},\chi(\cal{O}_X))\cdot ch(\cal{E})\}_3\nonumber\\
&=&p_*\{q^*(0,-rH,c_1\cdot H)\cdot ch(\cal{E})\}_3\nonumber\\
&=&p_*\{-rq^*H\cdot ch_2(\cal{E})+c_1(\cal{E})\cdot q^*(c_1\cdot H)\}\nonumber
\end{eqnarray}

Therefore
\begin{eqnarray}\label{computation}
c_1(\lambda(u_1))-c_1(\cal{E}xt^1_p(\cal{E},\cal{E}))&=&c_1(\lambda(u_1))+c_1(\cal{E}xt^{\bullet}_p(\cal{E},\cal{E}))\nonumber\\
&=&p_*\{-\frac{1}{2}c_1^2(\cal{E})\cdot q^*H+c_1(\cal{E})\cdot q^*(c_1\cdot H)\}
\end{eqnarray}

If we write $c_1(\cal{E})=p^*(c_1(Q))+q^*(c_1)$ for some line bundle $Q$ on $M(v)$, then (\ref{computation}) becomes
\begin{eqnarray}p_*\{-\frac{1}{2}(p^*c_1(Q)+q^*c_1^2+2p^*c_1(Q)q^*c_1)\cdot q^*H+(p^*c_1(Q)+q^*c_1)\cdot q^*(c_1\cdot H)\}=p_*\{-\frac{1}{2}p^*c_1^2(Q)\cdot q^*H\}\nonumber
\end{eqnarray}
and clearly
$$p_*(-\frac{1}{2}p^*c_1^2(Q)\cdot q^*H)=0,$$
as desired.

Combining J. Li 's result in \cite{L1} which implies that $\cal{L}_1$ is big and nef and section 3.3, we have

\begin{prop} Suppose the Gieseker moduli space $M_H(v)$ of primitive topological type $v$ on a smooth Del Pezzo surface $X$ is non-empty and irreducible, then $M_H(v)$ is a smooth weak Fano variety. In particular, it is a Mori dream space.

\end{prop}

\subsection{The case of $\bb{P}^2$}

A lot more is known in the case $X=\bb{P}^2$. First we have an explicit description on $v$ such that $M_H(v)$ is nonempty. Secondly, $M_H(v)$ is always irreducible and locally factorial and we have a explicit description of its Picard group thanks to the work of Drezet and Le Potier \cite{D}, \cite{DLP}. These properties allow us to remove the assumption that $v$ is of primitive type.  In this subsection, we summarize these properties of $M_H(v)$ (for arbitrary $v$), which will be used in this paper and refer to \cite{LP} for proofs.

\begin{definition}A (semi)stable sheaf $F$ on $\bb{P}^2$ is called (semi)exceptional if 
$$\rm{Ext}^1(F,F)=0.$$
\end{definition}

 All exceptional sheaves are bundles and if there exists an exceptional bundle of topological type $v$ then $M_H(v)$ is reduced to a point.

Drezet and Le Potier \cite{DLP} gave a necessary and sufficient condition for the existence of semistable torsion free sheaves of fixed topological type. They constructed a function $\delta:\rightarrow[\frac{1}{2},1]$ which is periodic of periodic of period $1$ and Lipschitz-continuous. We refer to \cite{LP} chapter 16 for the precise definition of $\delta$.

The main result is that

\begin{theorem}$($\cite{DLP}$)$ The necessary and sufficient condition for the existence of non-exceptional semistable sheaf of slope $\mu$ and discriminant $\Delta=c_1^2-2rch_2$ is
$$\Delta\ge\delta(\mu).$$
\end{theorem}

If $M_H(v)$ is nonempty, we also know the following properties,

\begin{theorem} $($\cite{D}$)$$($\cite{DLP}$)$ If nonempty, $M_H(v)$ is always irreducible, normal and locally factorial.


If $\Delta>\delta(\mu)$, the complement of $M^s_H(v)$ in $M_H(v)$ is of codimension at least $2$ and  the homomorphism in definition \ref{def3.1}
$$\lambda:v^{\perp}\rightarrow Pic(M^s_H(v))\cong Pic(M_H(v))$$
 is an isomorphism.
 
 If $\Delta=\delta(\mu)$, $Pic(M_H(v))$ is free abelian group of rank $1$.
\end{theorem}

\begin{remark} If $\Delta=\delta(\mu)$, it could happen that the complement of $M^s_H(v)$ in $M_H(v)$ is of codimension $1$. Nevertheless, the homomorphism $\lambda$ is still well defined and is an epimorphism. (c.f \cite{LP} 18.3.)
\end{remark}

If $\Delta=\delta(\mu)$,  $M_H(v)$ is automatically a Mori dream space since its Picard group is free abelian of rank 1. We will be mostly interested in the $\Delta>\delta(\mu)$ case (i.e Picard number is $2$). Although $M_H(v)$ may not be smooth, it is locally factorial and the complement of $M^s_H(v)$ has codimension at least $2$. So the calculation of canonical class of $M^s_H(v)$ in section 3.3 will extend to $M_H(v)$ and therefore $M_H(v)$ is still a Mori dream space.

Running a directed MMP on a Mori dream space $M$ of Picard number $2$ is straightforward. Recall that the movable cone
$\overline{Mov}(M)\subset N^1(M)_{\bb{R}}$ is (the closure of) the cone spanned by all line bundles $L$ whose stable base locus is of codimension at least $2$ in $M$. We have
$$Nef(M)\subset \overline{Mov}(M)\subset \overline{NE}^1(M).$$
For any big divisor $D$, draw a line connecting $D$ with an ample divisor $A$ on $M$. This line will cross finitely many walls between Mori chambers. The wall crossing corresponds to two different birational models of $M$. There are two cases. If the wall lies in the interior of $\overline{Mov}(M)$, then the corresponding birational map $M_i\dashrightarrow M_j$ is a $D$-flip. If the wall happen to be one boundary of $\overline{Mov}(M)$ (but not on the boundary of $\overline{NE}^1(M)$), it corresponds to a divisorial contraction 
$$M_j\rightarrow M_j'$$
on some $M_j$ contracting some irreducible divisor $B$ and there is just one more Mori chamber outside of $\overline{Mov}(M)$ generated by (the pull-back to $M$ of) $Nef(M'_j)$ and (the strict transform of) $B$. Thus if $D\in \overline{Mov}(M)$, after finitely many $D$-flips we ended up with a model birational model $M_j$ on which (the strict transform of) $D$ is semiample, or if $D\notin \overline{Mov}(M)$, after finitely many $D$-flips, we have a divisorial contraction
$$\pi: \ M_j\rightarrow M_j'$$
such that $D\in\pi^*(Nef(M_j'))+B$.

\section{Determinant line bundles on the Bridgeland moduli}
Now let $X=\bb{P}^2$. We will assume still $v=(r,c_1,ch_2)$ is primitive and think of the Chern characters as numbers. The potential wall associated to a pair of Chern characters:
$$v=(r,c,d)\ \text{and } v'=(r',c',d')$$
is the following subset of the upper-half plane:
$$W_{v,v'}:=\{(s,t)\ |\ \mu_{s,t}(v)=\mu_{s,t}(v')\}$$
where $\mu_{s,t}$ is the slope function defined in section 2. Specifically,

$$\mu_{s,t}(r,c,d)=\frac{\frac{-t^2}{2}r+(d-sc+\frac{s^2}{2}r)}{t(c-sr)}.$$

so the wall is given by:

$$W_{v,v'}=\{(s,t)|(s^2+t^2)(rc'-r'c)-2s(rd'-r'd)+2(cd'-c'd)=0\}$$

We will only be interested in walls in the region where $s<\frac{c_1\cdot H}{r}$, because a Mumford-stable
torsion-free sheaf $E$ of degree $c_1$ only belongs to the category $\cal{A}_s$ if $s < \frac{c_1\cdot H}{r}$. The potential walls are nested semi circles in this region. 

Walls are significant because if $E$, $F$ are objects of $\cal{D}^b(Coh(\bb{P}^2))$ with
$$ch(F)=v\ \text{and}\ ch(E)=v'$$
and if $F$ is a sub-object of $E$ in $\cal{A}_s$ with $(s,t)\in W_{v,v'}$, then $E$ is not $(s,t)$-stable by definition, but it {\bf may} be $(s,t)$-semistable and stable for nearby points on one side of the wall, but not the other. Crossing the wall would therefore change the set of $(s,t)$-stable objects.

Since the $(s,t)$-moduli of fixed topological type is constant along any potential wall and each potential wall will meet a 'quiver region' (c.f \cite{ABCH} section 7), every $(s,t)$-moduli space is a quiver moduli with respect to some polarization and for each $(s,t)$-(semi)stable objects $E$ of topological type $v$, either $E$ or $E[1]$ lies in the quiver heart
$$\cal{A}(k):=\langle\cal{O}_{\bb{P}^2}(k-2)[2],\ \cal{O}_{\bb{P}^2}(k-1)[1],\ \cal{O}_{\bb{P}^2}(k)\rangle$$ for suitable $k$. 
So either $E$ or $E[1]$ is of the form
$$\cal{O}_{\bb{P}^2}(k-2)^{n_0}\rightarrow\cal{O}_{\bb{P}^2}(k-1)^{n_1}\rightarrow\cal{O}_{\bb{P}^2}(k)^{n_2}$$

The transition from $\vec{n}$ to $\pm v$ is given by the matrix

 $$\left[
 \begin{matrix}
 1&-1&1\\
 k-2&-(k-1)&k\\
 \frac{(k-2)^2}{2}&\frac{-(k-1)^2}{2}&\frac{k^2}{2}
 \end{matrix}
 \right]
 $$

For each $(s,t)$ on the potential wall $W_{v,v'}$ (there could be more than one $v'$ giving the same wall, but $span\{v',v\}$ is determined by $(s,t)$), we associate a {\bf canonical} determinant line bundle on the Bridgeland moduli in the following manner.

Let $w_{s,t}$ be an integral topological type (up to scalar) in $H^*(X,\bb{R})_{alg}$ perpendicular to $v$ and $v'$ under the non-degenerate pairing $\chi$. Since on the plane $P=span\{v',v\}$, $\mu_{s,t}$ is constant, we can choose an orientation of $w_{s,t}$ such that $\chi(w_{s,t},ch(C))>0$ for any objects $C\in\cal{A}_s$ with $\mu_{s,t}(C)<\mu_{s,t}(v)$. Finally, choose a complex $F_{s,t}$ such that $ch(F_{s,t})=w_{s,t}$.

For a flat family of $(s,t)$-semistable complexes $\cal{E}$ of topological type $v$ (or $-v$) on $\bb{P}^2\times S$: 
\begin{eqnarray}\label{quiver}
U(k-2)\longrightarrow V(k-1)\longrightarrow W(k)
\end{eqnarray}
either $\cal{E}|_b\in\cal{A}_s$ or $\cal{E}[-1]|_b\in \cal{A}_s$ for any $b\in S$.

 We associate a line bundle on $S$ the same way as in definition \ref{def3.1}:
$$\lambda_{s,t}:=det(p_!([q^*F_{s,t}]\cdot[\cal{E}])) \text{ if } \cal{E}|_b\in \cal{A}_s$$
or
$$\lambda_{s,t}:=det(p_!([q^*F_{s,t}]\cdot[\cal{E}[-1]]) \text{ if } \cal{E}[-1]|_b\in \cal{A}_s$$

\begin{lemma}\label{lemma4.1} If $\cal{E}|_b\in\cal{A}_s$ for any $b\in S$, then 
$$\lambda_{s,t}=(\wedge^{n_0}U)^{a_0}\otimes(\wedge^{n_1}V)^{a_1}\otimes(\wedge^{n_2}W)^{a_2}$$
for some $\vec{a}=(a_0,a_1,a_2)$ and $a_i=(-1)^i\ \chi(w_{s,t},\cal{O}_{\bb{P}^2}(k-2+i))$.
\end{lemma}

\begin{proof} We have
\begin{eqnarray}p_!([q^*F_{s,t}]\cdot[U(k-2)])&=&p_!([q^*F_{s,t}\otimes q^*\cal{O}_{\bb{P}^2}(k-2)\otimes p^*U])\nonumber\\
&=&[U\otimes p_!(q^*F_{s,t}(k-2)])]\nonumber
\end{eqnarray}
and
$$[p_!(q^*(F_{s,t}(k-2))]=[\cal{O}_S^{\oplus a_0}]$$
So the first statement is clear.
To determine $a_0$, we apply Grothendieck-Riemann-Roch,
\begin{eqnarray}a_0&=&ch_0(p_!(q^*(F_{s,t}(k-2)))\nonumber\\&=&p_*\{q^*ch(F_{s,t})\cdot q^*ch(\cal{O}_{\bb{P}^2}(k-2))\cdot q^*td(T_{\bb{P}^2})\}_2\nonumber\\
&=&p_*\{(q^*(w_{s,t}\cdot ch(\cal{O}_{\bb{P}^2}(k-2))\cdot td(T_{\bb{P}^2}))\}_2\nonumber\\
&=&\chi(w_{s,t},\cal{O}_{\bb{P}^2}(k-2))\nonumber
\end{eqnarray}

One can similarly prove the formula for the $V$, $W$ terms.

\end{proof}

\begin{remark} If fiber of $\cal{E}[-1]$ is in $\cal{A}_s$,, we have an extra negative sign in front of the formula for $a_i$ due to the choice of the orientation of $w_{s,t}$.

\end{remark}

Lemma \ref{lemma4.1} implies that 
$$a_0n_0+a_1n_1+a_2n_2=\chi(w_{s,t},[\cal{O}_{\bb{P}^2}(k-2)^{n_0}\rightarrow\cal{O}_{\bb{P}^2}(k-1)^{n_1}\rightarrow\cal{O}_{\bb{P}^2}(k)^{n_2}])=\chi(w_{s,t},\pm v)=0$$
and therefore we can talk about quiver stable objects of dimension vector $\vec{n}$ with respect to polarization $\vec{a}$.

\begin{lemma}\label{lemma4.2} If $F\in\cal{A}_s$, then $\chi(w_{s,t},ch(F))>0$ (resp. $<0$) if and only if $\mu_{s,t}(F)<(resp. >)\ \mu_{s,t}(v)$.
\end{lemma}

\begin{proof}The plane $span\{v,v'\}$ in $H^*(X,\bb{Q})_{alg}$ is where $\mu_{s,t}$ equals constant $\mu_{s,t}(v)$. The conclusion follows from  the  choice of orientation of $w_{s,t}$.
\end{proof}
We prove
\begin{prop} An object $E=[\cal{O}_{\bb{P}^2}(k-2)^{n_0}\rightarrow\cal{O}_{\bb{P}^2}(k-1)^{n_1}\rightarrow\cal{O}_{\bb{P}^2}(k)^{n_2}]$ is quiver (semi)stable with respect to polarization $\vec{a}$ if and only if it is $(s,t)$-(semi)stable.
\end{prop}

\begin{proof}Without loss of generality, let us just prove the equivalence of semistable objects under both stability conditions. Suppose $E$ is $(s,t)$-semistable but quiver unstable. Since the quiver heart $\cal{A}(k)$ is coming from tilting $\cal{A}_s=\cal{P}_{s,t}(0,1]$ with respect to the torsion pair $(\cal{P}_{s,t}(\alpha,1],\cal{P}_{s,t}(0,\alpha])$ for some $\alpha\in(0,1]$, $E\in\cal{P}_{s,t}(\alpha,1]$ or $\cal{P}_{s,t}(0,\alpha][1]$. 

First let us assume $E\in\cal{P}_{s,t}(\alpha,1]$. 
Let $F$ be a quiver destablizing subobject of $E$ in $\cal{A}(k)$ with dimension vector $\vec{b}$. Then $\vec{a}\cdot\vec{b}<0$, or equivalently, $\chi(w_{s.t},ch(F))<0$. Since $Hom(\cal{P}_{s,t}(0,\alpha][1],\cal{P}_{s,t}(\alpha,1])=0$, $F\in\cal{P}_{s,t}(\alpha,1]$ as well. Let $F'\in\cal{P}_{s,t}(\alpha,1]\subset\cal{A}(k)$ be the first $(s,t)$-semistable factor of $F$, then by lemma \ref{lemma4.2}, $\mu_{s,t}(F')\ge\mu_{s,t}(F)>\mu_{s,t}(E)$, since $E$ is $(s,t)$-semistable, the composition
$$F'\rightarrow F\rightarrow E$$
is zero. This contradicts with the fact that $F$ is a subobject of $E$ in $\cal{A}(k)$. 

If $E\in\cal{P}_{s,t}(0,\alpha][1]$, we could consider a quiver destabilizing quotient $H$ of $E$ in $\cal{A}(k)$. A similar argument as above gives $H\in\cal{P}_{s,t}(0,\alpha][1]$ as well. The rest of the argument can be done similarly to the previous case by considering the last $(s,t)$-semistable factor $H'$ of $H$.

Conversely, suppose $E$ is quiver semistable. Notice that the $(s,t)$-phase $\phi_{s,t}(E)$ of $E$ is a well defined number in $(\alpha,\alpha+1]$ (although we do not know if $E$ is $(s,t)$-semistable, we can still talk about the phase of $Z_{s,t}(E)$). Write $E$ uniquely as an extension
$$P[1]\rightarrow E\rightarrow Q$$
where $P\in\cal{P}_{s,t}(0,\alpha]$, and $Q\in\cal{P}_{s,t}(\alpha,1])$. I claim that if $\alpha<\phi_{s,t}(E)\le1$, then $P=0$ whereras if $1<\phi_{s,t}(E)\le1+\alpha$, $Q=0$. For the first case, let $\vec{b}$ be the dimension vector of $P[1]$. Since $E$ is quiver semistable, $\vec{a}\cdot\vec{b}\ge0$, and therefore by lemma \ref{lemma4.1}, $\chi(w_{s,t},ch(P[1]))\ge0$ or equivalently $\chi(w_{s,t},ch(P))\le0$, by lemma \ref{lemma4.2}, this implies that $\phi_{s,t}(P)\ge\phi_{s,t}(E)$, a contradiction unless $P=0$. The second case can be treated similarly. Notice that in the second  case $ch(E)=-v$. If $Q\ne0$, let $\vec{c}$ be the dimension vector of $Q$. Then $\vec{a}\cdot\vec{c}\le0$, or equivalently, $\chi(w_{s,t},ch(Q))\ge0$ because in this case there is an extra negative sign in the formula in lemma \ref{lemma4.1}. This implies that $\phi_{s,t}(Q)\le\phi_{s,t}(E)-1$, contradiction.

  So again either $E\in\cal{P}_{s,t}(\alpha,1]$ or $E\in\cal{P}_{s,t}(0,\alpha][1]$. Suppose $E\in\cal{P}_{s,t}(\alpha,1]$, the other case can be treated similarly. Let $E'\in\cal{P}_{s,t}(\alpha,1]$ be the first $(s,t)$-semistable factor of $E$ with dimension vector $\vec{b}'$. Form the exact triangle
  \begin{eqnarray}\label{trig}E'\rightarrow E\rightarrow H
  \end{eqnarray}
  then $H\in\cal{P}_{s,t}(\alpha,1]$. Thus (\ref{trig}) is an exact sequence in $\cal{A}(k)$. Since $E$ is quiver semistable, $\vec{a}\cdot\vec{b}'\ge0$, again by lemma \ref{lemma4.2}, $\mu_{s,t}(E')\le\mu_{s,t}(E)$, so $E'=E$ is semistable.

\end{proof}

\section{The Bridgeland moduli as birational models of $M_H(v)$}

Notation as last section. The determinant line bundle $\lambda_{s,t}$ is always ample on the Bridgeland moduli space $M_{s,t}(v)$. Let $U\subset M_H(v)$ be an open subset of codimention at least $2$ and $\cal{E}$ be a flat family of sheaves on $X\times U$ which is both $(s,t)$-stable and Gieseker stable, then by construction of determinant line bundles,
$$\lambda_{s,t}\cong\lambda_{\cal{E}}(w_{s,t})$$
as line bundles on $U$.

This means the ample determinant line bundle $\lambda_{s,t}$ ton $M_{s,t}(v)$ pulls back to $\lambda(w_{s,t})$ on $M_H(v)$. Denote $M_{s,t}^P$ the normalization of the main component of $M_{s,t}(v)$ (with reduced induced scheme structure) whose generic point corresponds to a sheaf and still denote $\lambda_{s,t}$ as the restriction to it of the ample determinant line bundle. Then we can interpret 
$M_{s,t}^P(v)$
as birational models of $M_H(v)$.

More precisely, consider the Bridgeland wall and chamber structure with respect to the main component in the $(s,t)$-plane for $t>0$, $s<\frac{c_1\cdot H}{r}$. Since the walls are nested semi circles, we could choose an $s$ such that the ray $\{(s,t)|t>0\}$ intersects all {\bf{actual}} walls. Then we can decrease the parameter $t$, then $\lambda(w_{s,t})$ moves in $N^1(M_H(v))_{\bb{R}}$. An easy computation shows that $\lambda(w_{s,t})$ is moving toward the side of nef cone opposite to $-K_{M_H(v)}$. This corresponds to running a directed MMP on $M_H(v)$ and we get $M_{s,t}^P(v)$ as the birational models of $M_H(v)$.

When $t>>0$, it is proved in \cite{ABCH} that a sheaf $F$ is Gieseker stable if and only if it is $(s,t)$-stable. Thus 
$$M_{s,t}^P(v)\cong M_H(v)$$
for $t>>0$ and $\lambda_{s,t}\cong\lambda(w_{s,t})$ is ample on $M_H(v)$.

The first time $(s,t)$ hits an actual wall at $(s,t_0)$,  by the definition of wall, every $(s,t_0^+)$ semistable objects is still $(s,t_0)$ semistable, but some $(s,t_0^+)$-stable objects become strictly $(s,t_0)$-semistable. By the universal property of coarse moduli space, taking $S$-equivalence classes of $(s,t_0)$-semistable objects corresponds to a contracting morphism. It follows from lemma \ref{lemma5.1} that the exceptional loci is positive dimensional. (This is different from the case of sheaves on a K3 surface where there exists fake walls, i.e. the stable objects changed but the moduli space itself does not, see \cite{BM}.)  
$$\pi_0^+:\ M_H(v)\cong M_{s,t_0^+}(v)\longrightarrow M^P_{s,t_0}(v).$$
and $\lambda_{s,t_0}$ pulls back to a nef but not ample line bundle $\lambda(w_{s,t_0})$ on $M_H(v)\cong M_{s,t_0^+}(v)$. The first actual wall corresponds to one end of the nef cone of $ M_{s,t_0^+}$ (the other end being generated by $\lambda(u_1)$).

There are several possibilities:
\begin{enumerate}
\item $\pi_0^+$ is a fiber contraction. The directed MMP stops. In this case we actually have a ${\bf collapsing \ wall}$, which means there are no semistable objects in this component whatsoever after crossing the wall.
$$M^P_{s,t_0^-}=\varnothing.$$ 
\item $\pi_0^+$ is a divisorial contraction contracting $\Delta$. 
Crossing the wall will not affect $M_H(v)\setminus \Delta$. If $A$ is the destabilizing sheaf of $E$ at $(s,t_0)$, i.e $\mu_{s,t_0^+}(A)<\mu_{s,t_0^+}(E)$ but $\mu_{s,t_0^-}(A)>\mu_{s,t_0^-}(E)$, and $E$ sits in an exact sequence in $\cal{A}_s$ 
\begin{eqnarray}\label{des}
A\longrightarrow E\longrightarrow B
\end{eqnarray}

Notice that $A$, $B$ has to be at least $(s,t_0)$-semistable, otherwise the wall would have been crossed earlier. If both $A$ and $B$ are stable, then crossing the wall amounts to replace the $(s,t_0^+)$-stable objects $E$
by $(s,t_0^-)$-stable objects $E'$ which sits in the 'reverse' extension
\begin{eqnarray}
B\longrightarrow E'\longrightarrow A
\end{eqnarray}

If  $A$ or $B$ is strictly $(s,t_0)$-semistable, we just have to iterate the above process for the stable factors of $A$ or $B$. In any case, by lemma \ref{lemma5.1},
$$\pi_0^-:\ M^P_{s,t_0^-}\rightarrow M^P_{s,t_0}$$ 
has to be a small contraction. Because $M^P_{s,t_0}$ is $\bb{Q}$-factorial, we must have
$$M^P_{s,t_0^-}\cong M^P_{s,t_0}.$$ (Unlike the case of sheaves on K3 surfaces, there is no bouncing wall \cite{BM}, i.e  $M^P_{s,t_0^-}\cong M^P_{s,t_0^+}$ can not happen.) Since $M^P_{s,t_0^-}$ is of Picard number $1$,  the next wall has to be a collapsing wall.
\item $\pi_0^+$ is a small contraction. Again by lemma \ref{lemma5.1}, $\pi_0^-$ is a small contraction as well. The ample line bundle $\lambda_{s,t_0^-}$ on $M^P_{s,t_0^-}$ pulls back to $\lambda(w_{s,t_0^-})$ on $M_H(v)$. Therefore
$$M^P_{s,t_0^-}\cong Proj R(M^P_{s,t_0^-},\lambda_{s,t_0^-}))\cong Proj R(M_H(v),\lambda(w_{s,t_0^-}))$$
we get a flip
$$
\begin{diagram}
\node{M^P_{s,t_0^+}}\arrow{se,b}{\pi_0^+}\arrow[2]{e,--}\node{}\node{M^P_{s,t_0^-}}\arrow{sw,b}{\pi_0^-}\\
\node{}\node{M^P_{s,t_0}}\node{}
\end{diagram}
$$
 After crossing the wall, we can keep running the directed MMP by decreasing $t$ further and repeating the above process. As long as a generic point $E$ in the exceptional loci of $\pi_0^+$ is a sheaf, the destabilizing object $A$ is a sheaf as well. This follows from the long exact sequence in cohomology of (\ref{des}): 
 $$
 \begin{diagram}\dgARROWLENGTH=1.5em
 \node{H^{-1}(A)}\arrow{e,J}\node{H^{-1}(E)}\arrow{e}\node{H^{-1}(B)}\arrow{e}\node{H^0(A)}\arrow{e}\node{H^0(E)}\arrow{e,A}\node{H^0(B)}
 \end{diagram}
 $$
 The assumption in lemma \ref{lemma5.1} is still satisfied. Since there always exists a collapsing wall, as $t$ get small enough, the directed MMP will either ended up with case a) or b).

\end{enumerate}

\begin{lemma}\label{lemma5.1} Let $A$, $B$ be $(s,t_0)$-stable objects in $\cal{A}_s$ and $A\in\cal{Q}_s$ be a sheaf. Suppose that $\mu_{s.t_0}(A)=\mu_{s,t_0}(B)$ and $\mu_{s,t_0^+}(A)<\mu_{s,t_0^+}(B)$ (therefore $\mu_{s,t_0^-}(A)>\mu_{s,t_0^-}(B)$). Then $\dim_{\bb{C}} Hom_{\cal{D}}^1(B,A)>\dim_{\bb{C}} Hom_{\cal{D}}^1(A,B)$.
\end{lemma}
\begin{proof} By Serre duality,
$$Hom_{\cal{D}}^1(A,B)\cong Hom_{\cal{D}}^1(B,A\otimes^L\cal{O}(-3))^{\lor}.$$
Consider the exact sequence of sheaves on $\bb{P}^2$
$$0\rightarrow\cal{O}_{\bb{P}^2}(-3)\rightarrow\cal{O}_{\bb{P}^2}\rightarrow\cal{O}_C\rightarrow0$$
where $C$ is a general smooth cubic curve. Tensoring the above sequence with $B$ we get an exact triangle
$$B\otimes^L\cal{O}_{\bb{P}^2}(-3)\rightarrow B\rightarrow B\otimes^L\cal{O}_C$$
Applying the derived functor $\mathrm{RHom}^{\bullet}(A,-)$ and take long exact sequence in cohomology we obtain
$$
\begin{diagram}\dgARROWLENGTH=1.5em
\node{\cdots}\arrow{e}\node{Hom(A,B\otimes^L\cal{O}_{\bb{P}^2}(-3))}\arrow{e}\node{Hom(A,B)}\arrow{e}\node{Hom(A,B\otimes^L\cal{O}_C)}\node{}\\
\node{}\arrow{e}\node{Hom^1(A,B\otimes^L\cal{O}_{\bb{P}^2}(-3))}\arrow{e}\node{Hom^1(A,B)}\arrow{e}\node{Hom^1(A,B\otimes^L\cal{O}_C)}\node{}\\
\node{}\arrow{e}\node{Hom^2(A,B\otimes^L\cal{O}_{\bb{P}^2}(-3))}\arrow{e}\node{Hom^2(A,B)}\arrow{e}\node{Hom^2(A,B\otimes^L\cal{O}_C)}\arrow{e}\node{\cdots\hspace{2cm}}
\end{diagram}
$$

Since both $A$, $B$ are $(s,t_0)$-stable,
$$\mathrm{Hom}_{\cal{D}}(A,B)=0$$
and 
$$\mathrm{Hom}^2_{\cal{D}}(A,B\otimes^L\cal{O}_{\bb{P}^2}(-3))^{\lor}\cong\mathrm{Hom}_{\cal{D}}(B,A)=0.$$

Moreover I claim 
$$\mathrm{Hom}^i_{\cal{D}}(A,B\otimes^L\cal{O}_C)=0$$
for $i\le-2$ and $i\ge2$.

Assuming the claim, we get
\begin{eqnarray}\label{positive}
hom^1(B,A)-hom^1(A,B)&=&hom^0(A,B\otimes^L\cal{O}_C)-hom^1(A,B\otimes^L\cal{O}_C)\nonumber\\
&\ge&\chi(A^{\lor}\otimes^LB\otimes^L\cal{O}_C)\nonumber\\
&=&\int_{\bb{P}^2}ch(A^{\lor})\cdot ch(B)\cdot ch(\cal{O}_C)\cdot td(\bb{P}^2)\nonumber\\
&=&ch_0(A)ch_1(B)-ch_0(B)ch_1(A).
\end{eqnarray}

We prove that (\ref{positive}) has to be strictly positive. Notice that we have $\mu_{s,t}(A)<\mu_{s,t}(B)$ for $t>>0$, where
$$\mu_{s,t}(ch_0,ch_1,ch_2)=\frac{\frac{-t^2}{2}ch_0+(ch_2-sch_1+\frac{s^2}{2}ch_0)}{t(ch_1-sch_0)}$$
and for either $A$ or $B$, the denominator of $\mu_{s,t}$ is strictly positive. 
According to the rank of $A$, there are several cases:
\begin{enumerate}
\item $ch_0(A)=0$. Then $ch_0(B)\le0$ and $ch_1(A)>0$. If $ch_0(B)=0$ then the wall can never be crossed as we decrease $t$. If $ch_0(B)<0$ we immediately get (\ref{positive}) is positive.
\item $ch_0(A)>0$. Then we have an inequality for the dominant terms
$$\frac{ch_0(A)}{ch_1(A)-sch_0(A)}\ge\frac{ch_0(B)}{ch_1(B)-sch_0(B)}.$$
The equality can not be achieved because otherwise the wall can never be crossed again. But this precisely means $$ch_0(A)ch_1(B)-ch_0(B)ch_1(A)>0.$$
\end{enumerate}

It remains to prove the vanishing statement in the claim. Since $A$, $B$ are $(s,t_0)$-stable, we can assume
$A$, $B$ both are in the quiver heart 
$$\cal{A}(k)=\langle\cal{O}_{\bb{P}^2}(k-2)[2],\cal{O}_{\bb{P}^2}(k-1)[1],\cal{O}_{\bb{P}^2}(k)\rangle$$ for suitable $k$.
When $i\le-2$ or $i\ge3$, for degree reasons,
$$\mathrm{Hom}^i(\cal{O}_{\bb{P}^2}(k-j)[j],\cal{O}_C(k-j')[j'])=0$$
for any $j,j'=0,1,2$. This gives the vanishing statement for $i\le-2$ and $i\ge3$ case.

When $i=2$, $$\mathrm{Hom}^2(A,\cal{O}_C(k-j)[j])=0$$
is clear for $j=1,2$ because $A$ is a sheaf. We also have
$$\mathrm{Hom}^2_{\cal{D}}(A,\cal{O}_C(k))^{\lor}=\mathrm{Hom}_{\cal{O}_{\bb{P}^2}}(\cal{O}_C(k),A(-3))=0.$$
Because $\cal{O}_C(k)$ is a torsion sheaf, its image has to be torsion, but if $C$ is general there is no nontrivial map form $\cal{O}_C(k)$ to any fixed torsion sheaf.

\end{proof}

\section{Some examples in rank zero} 
We want to describe specifically the flips for the topological types $(0,4,-4)$ and $(0,5,-\frac{15}{2})$ by using the stratification of the Gieseker moduli given by Drezet and Maican \cite{DM} and by Maican \cite{M5} and give similar stratifications for the Bridgeland moduli spaces.\\
\\
Before starting to describe the birational models it is convenient to find an estimative that help us to bound the number of actual walls. Following same idea as in \cite{ABCH} for the case of the Hilbert scheme, if $\mathcal{A}$ is a sheaf of topological type $(0,c,d)$ with $c>0$ and $F$ is a destabilizing object (which is necessarily a sheaf) then $\mathcal{A}$ and $F$ fit into an exact sequence
$$
0\rightarrow K\rightarrow F\rightarrow \mathcal{A}
$$   
and by Corollary 6.2 in \cite{ABCH} we must have $K\in \mathcal{F}_s$ and $F\in\mathcal{Q}_s$ for all $s$ along the wall. If $ch(F)=(r',c',d')$ then in our case where the wall is a semicircle with center $(d/c,0)$ and radius $R$ this says
$$
\frac{d(K)}{r(K)}\leq \frac{d}{c}-R\ \ \ \text{ and }\ \ \ \ \frac{c'}{r'}\geq \frac{d}{c}+R. 
$$
Since $r(K)=r'$ and $d(K)-c'+c\geq 0$ then combining the inequalities above we get
\begin{equation}\label{bound1}
R\leq \frac{c'}{r'}-\frac{d}{c}\leq \frac{c}{r'}-R
\end{equation}
which immediately produces
\begin{equation}\label{boundradius}
R\leq \frac{c}{2r'}.
\end{equation}
\subsection{The birational geometry of $M_H^4=M_H(0,4,-4)$}
In \cite{DM} Drezet and Maican describe a stratification for the Gieseker moduli:
\begin{center}
\begin{tabular}{|c|c|c|}
\hline
Stratum & Exact sequence in $\mathcal{A}_{-1}$ & Codimension\\
\hline
& & 
\\
$X_0$ & $0\rightarrow \O(1)\rightarrow \mathcal{F}\rightarrow \O(-3)[1]\rightarrow 0$& $3$\\
& & 
\\
\hline
& & 
\\
$X_1$ & $0\rightarrow [\O(-1)\rightarrow 2\O]\rightarrow \mathcal{F}\rightarrow [2\O(-2)\rightarrow \O(-1)]\rightarrow 0$& $1$\\
& & 
\\
\hline
& & 
\\
$X_2$ & $0\rightarrow 2\O\rightarrow \mathcal{F}\rightarrow 2\O(-2)[1]\rightarrow 0$& $0$\\
& & 
\\
\hline
\end{tabular}
\end{center}
Since in the rank-zero case all the walls are semicircles with center $(d/c,0)$ then the obvious choice for $s$, in order to run the directed minimal model as explained before, is $s=d/c=-1$.\\
\\
The stratification above gives us all the Bridgeland walls:
\begin{center}
\begin{tabular}{|c|c|c|}
\hline
Wall $W_i$& Destabilizing Object & $R_i$\\
\hline
& & 
\\
$W_0$ & $ \O(1)$ & $2$\\
& &
\\
\hline
& & 
\\
$W_1$ & $[\O(-1)\rightarrow 2\O]$ & $\sqrt{2}$\\
& & 
\\
\hline
& & 
\\
$W_2$ & $2\O$ & $1$\\
& & 
\\
\hline
\end{tabular}
\end{center} 
$W_2$ is the collapsing wall and so any potential wall with radius $<1$ can not be an actual wall. The radius of a wall is given by
$$
R=\sqrt{1+2\frac{c'+d'}{r'}}
$$ 
where $(r',c',d')$ are the invariants of the destabilizing object producing the wall. From equation (\ref{boundradius}) we get
$$
R\leq \frac{2}{r'}
$$
and so a destabilizing element producing a wall other than the collapsing wall must have rank one. In this case, combining the formula for the radius and equation (\ref{boundradius}) we get that the invariants of a destabilizing object producing a wall must satisfy $2\chi'-c'=2,3,4,$ or $5$ producing walls with $R=1,\sqrt{2},\sqrt{3},2$ respectively. But the wall with $R=\sqrt{3}$ does not occur since in such case from equation (\ref{bound1}) we would have
$$
\sqrt{3}-1\leq c'\leq 3-\sqrt{3}
$$
and so $c'=1$ which is impossible since $2\chi'-c'=4$. Thus the only Bridgeland walls are the ones in the table above.\\
\\
Denote by $M_0,M_1$ and $M_2$ the Bridgeland moduli spaces at the walls $W_0,W_1$ and $W_2$ respectively and by $M_i^+$ and $M_i^-$ the moduli spaces for $t=R_i\pm \epsilon$ for some small enough $\epsilon>0$. Let $\pi_i:M_i^+\dashrightarrow M_i^-$ be the corresponding maps and $E_i^+,E_i^-$ the exceptionals for $\pi_i$ and $\pi_i^{-1}$. We have $\pi_0$ is a flip, $\pi_1$ is a divisorial contraction and $\pi_2$ is the contraction to a point. \\
\\
Crossing $W_0$ will produce complexes $F^{\bullet}$ that fit into an exact sequence
$$
0\rightarrow \O(-3)[1]\rightarrow F^{\bullet}\rightarrow \O(1)\rightarrow 0
$$  
such extensions are parametrized by $\Ext(1,\O(1),\O(-3)[1])$ and so $E_0^-$ is isomorphic to $\P^2$. To see how $E_0^{-}$ intersects $E_1^{+}$ notice that for every $p \in \P^2$ there is a unique nontrivial class of extensions of the form
$$
0\rightarrow \O(-3)[1]\rightarrow \mathcal{G}\rightarrow \C_p\rightarrow 0.
$$ 
Pulling back such extension will produce a diagram of the form
$$
\begin{diagram}\dgARROWLENGTH=1.5em
\node{}\node{}\node{}\node{I_p(1)}\arrow{s}\arrow{sw,--}\node{}\\
\node{0}\arrow{e}\node{\O(-3)[1]}\arrow{e}\arrow{s,=}\node{F^{\bullet}}\arrow{e}\arrow{s}\node{\O(1)}\arrow{e}\arrow{s}\node{0}\\
\node{0}\arrow{e}\node{\O(-3)[1]}\arrow{e}\node{\mathcal{G}}\arrow{e}\node{\C_p}\arrow{e}\node{0}
\end{diagram}
$$
and so $F^{\bullet}\in E_0^-\cap E_1^+$. Since the elements we produce in this way are parametrized by $\P^2$ and $E_0^-\cong\P^2$ then we conclude that $E_0^-\subset E_1^+$.\\
\\
Drezet and Maican have also proved that if $N(6,2,2)$ denotes the set of semi-stable morphisms $2\O(-2)\rightarrow 2\O$ then $X=X_1\cup X_2$ is an open set of the blow up $\widetilde{N}$ of $N(6,2,2)$ along $\P^2\times \P^2$. Moreover, the complement $\widetilde{N}\setminus X$ is isomorphic to $\P^2$ and indeed it coincides with $E_0^-$. Now, crossing $W_1$ will produce complexes that are extensions of the form
$$
0\rightarrow I_q^{\vee}(-3)[1]\rightarrow F^{\bullet}\rightarrow I_p(1)\rightarrow 0
$$
but since $\ext(1,I_p(1),I_q^{\vee}(-3)[1])=1$ then $E_1^{-}$ is isomorphic to $\P^2\times \P^2$. This proves that $M_1^{+}\cong \widetilde{N}$, $M_1^{-}\cong N(6,2,2)$ and the divisorial contraction $\pi_1:M_1^+\dashrightarrow M_1^{-}$ can be extended to a morphism and it is in fact the blow up of $N(6,2,2)$ along $\P^2\times\P^2$. 
\subsection{The birational geometry of $M_{H}^5=M_{H}(0,5,-\frac{15}{2})$}
The Maican stratification of the Gieseker moduli is:
\begin{center}
\begin{tabular}{|c|c|c|}
\hline
Stratum & Exact sequence in $\mathcal{A}_{-3/2}$ & Codimension\\
\hline
& & 
\\
$X_0$ & $0\rightarrow \O(1)\rightarrow \mathcal{F}\rightarrow \O(-4)[1]\rightarrow 0$& $6$\\
& & 
\\
\hline
& & 
\\
$X_1$ & $0\rightarrow [\O(-1)\rightarrow 2\O]\rightarrow \mathcal{F}\rightarrow [2\O(-3)\rightarrow \O(-2)]\rightarrow 0$& $4$\\
& & 
\\
\hline
& & 
\\
$X_2$ & $0\rightarrow [2\O(-2)\rightarrow 2\O(-1)\oplus \O]\rightarrow \mathcal{F}\rightarrow \O(-3)[1]\rightarrow 0$& $1$\\
& & 
\\
\hline
& & 
\\
$X_3$ & $0\rightarrow 5\O(-1)\rightarrow \mathcal{F}\rightarrow 5\O(-2)[1]\rightarrow 0$ & 0\\
& & 
\\
\hline
\end{tabular}
\end{center}
The stratum $X_3$ is an open set, $X_0$ is closed and  $X_1,X_2$ are locally closed subsets. The idea is that each strata produce a Bridgeland wall and that those are the only walls that change the main component of the Bridgeland moduli. \\
\\
As before, the walls are semicircles with center $(-\frac{3}{2},0)$. The radius of a Bridgeland wall corresponding to a destabilizing object with invariants $(r',c',d')$ is:
$$
R=\sqrt{\frac{9}{4}+\frac{2d'+3c'}{r'}}.
$$
Again, the natural choice for the ray that will give us the directed minimal model is $(-\frac{3}{2},t)$. It is easy to check that the sub objects in the Maican stratification are sub objects in the $\mathcal{A}_{-3/2}$ category and so they produce the following walls:
\begin{center}
\begin{tabular}{|c|c|c|}
\hline
Wall & Destabilizing Object & $R$\\
\hline
& & 
\\
$W_0$ & $ \O(1)$ & $\frac{5}{2}$\\
& &
\\
\hline
& & 
\\
$W_1$ & $[\O(-1)\rightarrow 2\O]$ & $\frac{\sqrt{17}}{2}$\\
& & 
\\
\hline
& & 
\\
$W_2$ & $ [2\O(-2)\rightarrow 2\O(-1)\oplus \O]$ & $\frac{3}{2}$\\
& & 
\\
\hline
& & 
\\
$W_3$ & $5\O(-1)$ & $\frac{1}{2}$\\
& & 
\\
\hline
\end{tabular}
\end{center} 
Then it is clear that $W_3$ is the collapsing wall since all the objects in the open subset $X_3$ get replaced by complexes that are extensions of the form
$$
0\rightarrow 5\O(-2)[1]\rightarrow F^{\bullet}\rightarrow 5\O(-1)\rightarrow 0
$$
and all these extensions are trivial.\\
\\
To see that these are the only interesting walls (walls that produce all the birational models of $M^5_{H}$) we use equation (\ref{boundradius}) to conclude that if $W$ is a rank $r'$ wall then
$$
R\leq \frac{5}{2r'} 
$$
and therefore there are no walls  with $r'>5$ since such wall would be contained inside the collapsing wall.  In our case
$$
R=\sqrt{\frac{9}{4}+2\frac{\chi'}{r'}-2}\leq \frac{5}{2r'}
$$
where $\chi'$ denotes the Euler characteristic of a destabilizing sheaf of invariants $(r',c',d')$. Thus if $r'=2$ we get $\chi'=1$, if $r'=3$ we get $\chi'=0$ or 1, if $r'=4$ or $5$ then $\chi'=0$. But if $\chi'=0$ then $R=1/2$ and we get the collapsing wall. For $\chi'=1$ we have
$$
R=\begin{cases}\frac{\sqrt{5}}{2}& \text{if   }r'=2\\
\sqrt{\frac{11}{12}} &\text{if   }r'=3\end{cases}
$$ 
Notice that a rank 3 wall would occur after the divisorial contraction and so it must be the collapsing wall instead of $W_3$ or it destabilizes a component of the Bridgeland moduli other than the main component. Rank 2 walls do not occur either since by equation (\ref{bound1}) we would have
$$
 -3+\sqrt{5}\leq c'\leq 2-\sqrt{5}
$$
which is impossible being $c'$ an integer.\\
\\
Further computations easily show that the only rank one walls are $W_0,W_1,W_2$ and $W_3$ that can be seen also as a rank one wall for the destabilizing object $\O(-1)$.\\
\\ 
We have two flips corresponding to crossing the walls $W_0$ and $W_1$ and a divisorial contraction corresponding to crossing $W_2$.  Let $M_0,M_1,M_2$ and $M_3$ be the Bridgeland moduli spaces at the walls and denote by $M_i^{+}$ and $M_i^{-}$ the Bridgeland moduli spaces for nearby $t$ as before. Let $\pi_i^{\pm}:M_i^{\pm}\rightarrow M_i$ denote the corresponding contractions and $E_i^{+}, E_i^{-}$ the exceptional loci for the birational maps $\pi_i:=(\pi_i^-\circ\pi_i^+):M_i^{+}\dashrightarrow M_i^{-}$ and $\pi_i^{-1}:M_i^{-}\dashrightarrow M_i^{+}$. We have:
\begin{itemize}
\item $E_0^{+}$ and $E_0^{-}$ are projective spaces. Indeed, $E_0^{+}$ is isomorphic to the moduli of plane quintics which can be identified with $\P^{20}$. Crossing the wall will produce extensions of the form
$$
0\rightarrow \O(-4)[1]\rightarrow F^{\bullet} \rightarrow \O(1)\rightarrow 0
$$
which are parametrized by $\P(\Ext(1,\O(1),\O(-4)[1]))\cong \P(H^0(\P^2,\O(2))^{\vee})=\P^5$.
\item $E_1^{+}$ and $E_1^{-}$ are projective bundles over $\P^2\times\P^2$. Indeed, every object in $E_1^+$ fits into an exact sequence
$$
0\rightarrow I_p(1)\rightarrow \mathcal{F}\rightarrow I_q^{\vee}(-4)[1]\rightarrow 0
$$
where $I_p$ and $I_q$ denote the ideal sheaves at points $p,q\in\P^2$ and $\vee$ stands for derived dual. Thus, $E_1^+$ is a projective bundle over $\P^2\times \P^2$ with fiber
$$
\P(\Ext(1,I_q^{\vee}(-4)[1],I_p(1)))
$$
which has dimension 18. Crossing the wall will produce complexes that are extensions of the form
$$
0\rightarrow I_q^{\vee}(-4)[1]\rightarrow F^{\bullet} \rightarrow I_p(1)\rightarrow 0
$$ 
which form again a projective bundle over $\P^2\times \P^2$ now with fiber
$$
\P(\Ext(1,I_p(1),I_q^{\vee}(-4)[1]))
$$
of dimension 3.
\item $E_2^{+}$ is a projective bundle over an appropriate model of the Hilbert scheme $Hilb^5(\P^2)$ and $E_2^{-}$ is isomorphic to such model. Indeed, elements in $E_2^+$ fit into an exact sequence
$$
0\rightarrow [2\O(-2)\rightarrow 2\O(-1)\oplus \O]\rightarrow \mathcal{F}\rightarrow \O(-3)[1]\rightarrow 0
$$
and the map $2\O(-2)\rightarrow 2\O(-1)\oplus \O$ is generically injective with quotient $I_5(2)$, here $I_5$ denotes the ideal sheaf of five points on $\P^2$, then the complex $[2\O(-2)\rightarrow 2\O(-1)\oplus \O]$ must lie in some of the models of $Hilb^5(\P^2)$. To see which model exactly we observe that at the wall $W_2$ such complex has to be semistable since otherwise any destabilizing object would also destabilize our element in $E_2^+$. Denote by $Hilb^5(\P^2)(2)$ the Hilbert scheme of five points where every element is an ideal sheaf twisted by $\O(2)$. Walls for $Hilb^5(\P^2)(2)$ have been described in \cite{ABCH}, here is a summary of their found:\\
\begin{center}
\begin{tabular}{|c|c|c|c|}
\hline
Wall & Dest. Object & Exceptional Locus  &$R$\\
\hline
& & &
\\
$W_{-\frac{7}{2}}$ & $\O(1) $& \begin{tabular}{cc}$\mathcal{E}_0^+:$ & $0\rightarrow\O(1)\rightarrow I_5(2)\rightarrow \O_{l}(-3)\rightarrow 0$ \\ $\mathcal{E}_0^-:$ & $0\rightarrow\O_{l}(-3)\rightarrow G^{\bullet} \rightarrow \O(1)\rightarrow 0$\end{tabular} & $\frac{9}{2}$\\
& & & \\
\hline
& & &
\\
$W_{-\frac{5}{2}}$ & $I_p(1) $& \begin{tabular}{cc}$\mathcal{E}_1^+:$ & $0\rightarrow I_p(1)\rightarrow I_5(2)\rightarrow \O_{l}(-2)\rightarrow 0$ \\ $\mathcal{E}_1^-:$ & $0\rightarrow\O_{l}(-2)\rightarrow G^{\bullet} \rightarrow I_p(1)\rightarrow 0$\end{tabular} & $\frac{\sqrt{41}}{2}$\\
& & & \\
\hline
& & &
\\
$W_{-\frac{3}{2}}$ &\begin{tabular}{c} $\O $\\  \\ \\ $I_2(1)$\end{tabular}& \begin{tabular}{cc}$\mathcal{E}_{2,1}^+:$ & $0\rightarrow \O \rightarrow I_5(2)\rightarrow [2\O(-2)\rightarrow 2\O(-1)]\rightarrow 0$ \\ &  \\ & \\ $\mathcal{E}_{2,2}^+:$ & $0\rightarrow I_2(1)\rightarrow I_5(2)\rightarrow \O_{l}(-1)\rightarrow 0$ \end{tabular} & $\frac{3}{2}$\\
& & & \\
\hline
\end{tabular}
\end{center}
Thus the collapsing wall for the Hilbert scheme coincides with $W_2$ and therefore the complex $[2\O(-2)\rightarrow 2\O(-1)\oplus \O]$ lies on the model $N_2^+$ of $Hilb^5(\P^2)(2)$ in the chamber determined by the walls $W_{-\frac{5}{2}}$ and  $W_{-\frac{3}{2}}$.\\
\\
This shows that $E_2^+$ is a projective bundle over $N_2^+$ with generic fiber $\P(\Ext(1,\O(-3)[1],I_5(2)))$ which has dimension 15. \\
\\
As explained in section 5, $M_2^-\cong M_2$ and moreover $\pi_2^+(E_2^+)$ is the set of all split extensions 
$$
0\rightarrow \mathcal{N}\rightarrow G \rightarrow \O(-3)[1]\rightarrow 0 
$$
where $\mathcal{N}$ is an element of the model $N_2$ of $Hilb^5(\P^2)(2)$ at the collapsing wall. The contraction $\pi_2^+|_{E_2^+}$ is the composition of the fiber contraction followed by the final contraction $N_2^+\rightarrow N_2$.   
\end{itemize} 
\subsubsection{The stratification of the Bridgeland moduli}
We have described the exceptional loci but in order to give stratifications similar to those of the Gieseker moduli we need to know how these loci intersect.
\begin{itemize}
\item $E_1^+$ intersects $E_0^-$ along the Veronesse surface. Indeed, $\P^2$ embeds naturally in $E_0^-=\P(\Ext(1,\O(1),\O(-4)[1]))\cong \P(H^0(\O(2))^{\vee})$ by the 2-uple embedding, its image is the set of extensions fitting in a commutative diagram
$$
\begin{diagram}\dgARROWLENGTH=1.5em
\node{}\node{}\node{}\node{I_p(1)}\arrow{sw,--}\arrow{s}\node{}\\
\node{0}\arrow{e}\node{\O(-4)[1]}\arrow{e}\node{F^{\bullet}}\arrow{e}\node{\O(1)}\arrow{e}\node{0}
\end{diagram}
$$
Such extensions come from pulling back  the unique nontrivial class of extensions in $\Ext(1,\C_p,\O(-4)[1])$ by the exact sequence $0\rightarrow I_p(1)\rightarrow \O(1)\rightarrow \C_p\rightarrow 0$ and therefore the Veronesse surface corresponds to a section of $E_1^+$ over the diagonal. For simplicity we denote by $Y$ the projective space $E_0^-$ and by $X$ the image of $\P^2$ by the 2-uple embedding. 
\item To find $E_1^-\cap E_2^+$ we recall that every $N\in \mathcal{E}_1^-$ fits into an exact sequence
$$
0\rightarrow \O_l(-2)\rightarrow N\rightarrow I_p(1)\rightarrow 0.
$$ 
By pushing forward extensions classes $[S]\in \Ext(1,\O(-3)[1],\O_l(-2))$ we get commutative diagrams 
$$
\begin{diagram}\dgARROWLENGTH=1.5em
\node{}\node{0}\arrow{s}\node{}\node{}\node{}\\
\node{0}\arrow{e}\node{\O_l(-2)}\arrow{s}\arrow{e}\node{S}\arrow{s}\arrow{e}\node{\O(-3)[1]}\arrow{e}\arrow{s,=}\node{0}\\
\node{0}\arrow{e}\node{N}\arrow{s}\arrow{e}\node{F^{\bullet}}\arrow{e}\arrow{sw,b,--}{h}\node{\O(-3)[1]}\arrow{e}\node{0}\\
\node{}\node{I_p(1)}\arrow{s}\node{}\node{}\node{}\\
\node{}\node{0}\node{}\node{}\node{}
\end{diagram}.
$$
It can be checked that $[S]=[I^{\vee}_q[-4]]$ for some $q\in l$ and so $E_1^+\cap E_0^-$ is a tower of projective bundles over the divisor on $\P^2\times (\P^2)^*$ consisting of points $(p,l)$ with $p\in l$.

\item More interesting is to describe the strict transform of $Y$ by $\pi_1$. We claim 
$$
\overline{\pi_1(Y\setminus X)}\cong Bl_{X}Y.
$$
Indeed, since $Y$ intersects $E_1^+$ along a section over the diagonal we have a closed immersion $Y\rightarrow M_1$ and a diagram
$$
\begin{diagram}\dgARROWLENGTH=1.5em
\node{}\node{Bl_{E_1^+}M_1^+}\arrow{s}\arrow{se}\node{}\\
\node{Bl_{X}Y}\arrow{ne}\arrow{s}\node{M_1^+}\arrow{se}\node{M_1^-}\arrow{s}\\
\node{Y}\arrow{ne,J}\arrow[2]{e,b,J}{\pi_1^+|_{Y}}\node{}\node{M_1}
\end{diagram}
$$
where every square is a fiber square, giving a closed immersion $Bl_{X}Y\rightarrow M_1^-$.\\
\\
To prove that $M_1^-\times_{M_1}M_1^+\cong Bl_{E_1^+}M_1^+\cong Bl_{E_1^-}M_1^-$, let $F^{\bullet}\in E_1^+$, i.e., $F^{\bullet}$ is an extension of the form:
$$
0\rightarrow I_{p}(1)\rightarrow F^{\bullet}\rightarrow I_q^{\vee}(-4)[1]\rightarrow 0.
$$
Applying $\Hom(F^{\bullet},\_)$ and $\Hom(\_,F^{\bullet})$ we get respectively diagrams
$$
\begin{diagram}\dgARROWLENGTH=1.0em
\node{}\node{}\node{0}\arrow{s}\node{}\\
\node{}\node{}\node{\Ext(1,I_q^{\vee}(-4)[1],I_q^{\vee}(-4)[1])}\arrow{s}\node{}\\
\node{\hspace{1cm}\Ext(1,F^{\bullet},I_p(1))}\arrow{e}\node{\Ext(1,F^{\bullet},F^{\bullet})}\arrow{se,b}{f}\arrow{e}\node{\Ext(1,F^{\bullet},I_q^{\vee}(-4)[1])}\arrow{s}\arrow{e}\node{\Ext(2,F^{\bullet},I_p(1))=0}\\
\node{}\node{}\node{\Ext(1,I_p(1),I_q^{\vee}(-4)[1])}\arrow{s}\node{}\\
\node{}\node{}\node{0}\node{}
\end{diagram}
$$
and 
$$
\begin{diagram}\dgARROWLENGTH=1.0em
\node{}\node{}\node{0}\arrow{s}\node{}\\
\node{}\node{}\node{\Ext(1,I_p(1),I_p(1))}\arrow{s}\node{}\\
\node{\hspace{1cm}\Ext(1,I_q^{\vee}(-4)[1],F^{\bullet})}\arrow{e}\node{\Ext(1,F^{\bullet},F^{\bullet})}\arrow{se,b}{f}\arrow{e}\node{\Ext(1,I_p(1),F^{\bullet})}\arrow{s}\arrow{e}\node{\Ext(2,I_q^{\vee}(-4)[1],F^{\bullet})=0}\\
\node{}\node{}\node{\Ext(1,I_p(1),I_q^{\vee}(-4)[1])}\arrow{s}\node{}\\
\node{}\node{}\node{0}\node{}
\end{diagram}
$$
Thus $V=ker(f)$ fits into an exact sequence
$$
\hspace{1.3cm}0\rightarrow \C\rightarrow \Ext(1,I_q^{\vee}(-4)[1],I_p(1))\rightarrow V\rightarrow\Ext(1,I_p(1),I_p(1))\oplus \Ext(1,I_q^{\vee}(-4)[1],I_q^{\vee}(-4)[1])\rightarrow 0 
$$
This gives the exact sequence
$$
0\rightarrow (TE_1^+)_{F^{\bullet}}\rightarrow (T M_1^+|_{E_1^+})_{F^{\bullet}}\rightarrow \Ext(1,I_p(1),I_q^{\vee}(-4)[1])=H^0(\mathbb{P}^2,I_p\otimes I_q(2))^*\rightarrow 0.
$$
Now, consider the fiber square
$$
\begin{diagram}\dgARROWLENGTH=1em
\node{\P^2\times\P^2\times\P^2}\arrow{e,t}{p_{23}}\arrow{s,l}{p_{12}}\node{\P^2\times\P^2}\arrow{s,r}{p_1}\\
\node{\P^2\times\P^2}\arrow{e,b}{p_2}\node{\P^2}
\end{diagram}
$$
with the obvious projections. Let $\mathcal{I}^+=p_{12}^*(\mathcal{I}_{\Delta})\otimes p_{23}^*(\mathcal{I}_{\Delta}(2,0))$ where $\mathcal{I}_{\Delta}$ is the ideal sheaf of the diagonal $\Delta\subset\P^2\times\P^2$. $\mathcal{I}^+$ is flat over $\P^2\times\P^2$ via the projection $p_{13}:\P^2\times\P^2\times\P^2\rightarrow \P^2\times\P^2$ since $\mathcal{I}_{\Delta}$ is flat over $\P^2$. Thus ${p_{13}}_{*}(\mathcal{I}^+)$ is locally free and the computation above shows that there is an exact sequence
$$
0\rightarrow TE_1^+\longrightarrow T M_1^+|_{E_1^+}\longrightarrow \left(\pi_1^+|_{E_1^+}\right)^{*}({p_{13}}_{*}(\mathcal{I}^+))^*\rightarrow 0.
$$
Similarly we can get an exact sequence
$$
0\rightarrow TE_1^-\longrightarrow T {M_1^-}|_{E_1^-}\longrightarrow \left(\pi_1^-|_{E_1^-}\right)^{*}({p_{13}}_{*}(\mathcal{I}^-))\rightarrow 0,
$$
where $\mathcal{I}^-=p_{12}^*(\mathcal{I}_{\Delta})\otimes p_{23}^*(\mathcal{I}_{\Delta}(5,0))$. This shows that the fiber product $M_1^-\times_{M_1}M_1^+$ restricts to a fibered diagram
$$
\begin{diagram}
\node{\P(\mathcal{N}_{E_1^+/M_1^+})= \P(\mathcal{N}_{E_1^-/M_1^-})}\arrow{e}\arrow{s}\node{E_1^-}\arrow{s,r}{\pi_1^-|_{E_1^-}}\\
\node{E_1^+}\arrow{e,b}{\pi_1^+|_{E_1^+}}\node{\P^2\times\P^2}
\end{diagram}
$$
and our claim follows.
\item If $l\subset \P^2$ is any line then $\ext(1,\O(1),\O_l(-3))=3$ and the natural map $\Ext(1,\O(1),\O_l(-3))\rightarrow \Ext(1,\O(1),\O(-4)[1])$ is injective. Let $H$ be the plane in $Y=\mathbb{P}^5$ obtained this way, i.e., the set of complexes fitting into a commutative diagram
$$
\begin{diagram}\dgARROWLENGTH=1.5em
\node{}\node{0}\arrow{s}\node{0}\arrow{s}\node{}\node{}\\
\node{0}\arrow{e}\node{\O_l(-3)}\arrow{s}\arrow{e}\node{N}\arrow{s}\arrow{e}\node{\O(1)}\arrow{e}\arrow{s,=}\node{0}\\
\node{0}\arrow{e}\node{\O(-4)[1]}\arrow{s}\arrow{e}\node{F^{\bullet}}\arrow{sw,--}\arrow{e}\node{\O(1)}\arrow{e}\node{0}\\
\node{}\node{\O(-3)[1]}\arrow{s}\node{}\node{}\node{}\\
\node{}\node{0}\node{}\node{}\node{}
\end{diagram}.
$$
We claim that the image of $l$ by the 2-uple embedding is contained in $H$. Indeed, if $p\in l$ then we get a diagram
$$
\begin{diagram}\dgARROWLENGTH=1.5em
\node{}\node{}\node{}\node{\O_l(-3)}\arrow{s}\arrow{sw,--}\node{}\\
\node{0}\arrow{e}\node{I_p(1)}\arrow{e}\arrow{s,=}\node{N}\arrow{e}\arrow{s}\node{\O_l(-2)}\arrow{e}\arrow{s}\node{0}\\
\node{0}\arrow{e}\node{I_p(1)}\arrow{e}\node{\O(1)}\arrow{e}\node{\C_p}\arrow{e}\node{0}
\end{diagram}
$$
which completes the first diagram implying $p\in X$. This implies that the closure of $X_2$ in $M_0^-$ intersects $E_0^- $ along $Sec\ X$, which is the union of all $H$ in $Y$ as $l$ varies in $\mathbb{P}^{2*}$. Thus
$$
E_2^+\cap \overline{\pi_1(Y\setminus X)}\cong \widetilde{Sec\ X}.
$$
\item Finally, we would like to know what happens to $\widetilde{Sec\ X}$ at the wall $W_2$. To see this, notice that off $X$, a point $F^{\bullet}$ in $Sec\ X$ corresponds to a non-split extension
$$
0\rightarrow N\rightarrow F^{\bullet}\rightarrow \O(-3)[1]\rightarrow 0
$$ 
with $N\in \mathcal{E}_0^{-}$. Such $N$ can be obtained as a pull back diagram 
$$
\begin{diagram}\dgARROWLENGTH=1em
\node{}\node{}\node{}\node{0}\arrow{s}\node{}\\
\node{}\node{}\node{}\node{\O}\arrow{s}\arrow{sw}\node{}\\
\node{0}\arrow{e}\node{\O_l(-3)}\arrow{e}\arrow{s,=}\node{N}\arrow{e}\arrow{s}\node{\O(1)}\arrow{s}\arrow{e}\node{0}\\
\node{0}\arrow{e}\node{\O_{l}(-3)}\arrow{e}\node{\mathcal{G}}\arrow{s}\arrow{e}\node{\O_l(1)}\arrow{s}\arrow{e}\node{0}\\
\node{}\node{}\node{0}\node{0}\node{}
\end{diagram}
$$
The final contraction $N_2^+\rightarrow N_2$ sends every such $N$ to $\O\oplus \mathcal{G}$ and so $\pi_2^+:M_2^+\rightarrow M_2$ contracts $\widetilde{Sec\ X}$ to a $(\P^2)^*$.\\
\\ 
This information is enough to write down stratifications by locally closed subsets for the Bridgeland moduli spaces. Moreover, each strata can be embedded as an open subset of a projective bundle.

\end{itemize}
\bibliographystyle{alpha}
\bibliography{references.bib}
\end{document}